\documentclass{amsart} \usepackage{amssymb,amsfonts,amsmath,graphicx}
\usepackage[alphabetic,y2k,lite]{amsrefs}
\usepackage{fullpage}
\usepackage{MnSymbol,float}
\usepackage[mathletters]{ucs}
\usepackage[utf8x]{inputenc}
\usepackage{cite}

\newtheorem{theorem}{Theorem}
\newtheorem{definition}[theorem]{Definition}
\newtheorem{proposition}[theorem]{Proposition}

\newcommand{\thmref}[1]{Theorem \ref{#1}}

\newcommand{\defref}[1]{Definition \ref{#1}}

\newcommand{\wdtld}[1]{\widetilde{#1}}

\newcommand{\vol}{{\textrm{vol}}}
\newcommand{\vd}{{\textrm{vd}}}
\newcommand{\vdm}{\wdtld{\vd}}
\newcommand{\wda}{\wdtld{a}}
\newcommand{\vtet}{{v_{\textrm{tet}}}}
\newcommand{\voct}{{v_{\textrm{oct}}}}

\newcommand{\falsum}{\wdtld{\#}}
\newcommand{\RR}{\mathbb{R}}

\author{Alice Kwon \textsuperscript{\textdagger}}
\author{Ying Hong Tham \textsuperscript{*}}

\address{\textsuperscript{\textdagger} \small Alice Kwon, Department of Science, SUNY Maritime, 6 Pennyfield Avenue, Bronx, NY 10465, USA}
\email{akwon@sunymaritime.edu}

\address{\textsuperscript{*} \small Ying Hong Tham
	Fachbereich Mathematik, Universit\"at Hamburg, Bundesstra{\ss}e 55, 20146 Hamburg, Germany} 
\email{ying.hong.tham@uni-hamburg.de}

\thanks{This work was supported by Deutsche Forschungsgemeinschaft via the Cluster of Excellence EXC 2121 ``Quantum Universe'' - 390833306.}

\makeatletter
\@namedef{subjclassname@2020}{%
  \textup{2020} Mathematics Subject Classification}
\makeatother

\begin{document}

\title{On the Volume Density Spectrum of Fully Augmented Links}
\maketitle

\begin{abstract}
For a hyperbolic fully augmented link in $S³$,
its \emph{FAL volume density} is the ratio of its volume
to the number of augmentations.
We show that the set of FAL volume densities
is dense in $[2\voct, 10\vtet)$, but discrete in $[\voct,2\voct)$.
\end{abstract}


Champanerkar, Kofman, and Purcell \cite{ckp-max} initiated the study of
the asymptotic behavior of the \emph{volume density} of hyperbolic links $K$,
defined as the ratio of volume to number of crossings (see \defref{d:vol-dens}),
as they Følner converge (see \cite[Def. 1.3]{ckp-max})
to some infinite biperiodic link.
In particular, they showed that for any sequence of hyperbolic links
that Følner converges to the infinite square weave,
their volume densities also converge to that of the infinite square weave,
which is $\voct$, the volume of the regular hyperbolic ideal octahedron.
This establishes that $\voct$ is a sharp upper bound on volume density.
Burton \cite{burton} proved that the set of volume densities
of all hyperbolic links is a dense subset of the interval $(0,\voct)$.

In this paper, we establish similar results for
hyperbolic \emph{fully augmented links (FALs)}.
Kwon \cite[Prop 3.6]{kwon2020} proves that the \emph{FAL volume density}
of a hyperbolic FAL $L$,
defined as the ratio of volume to number of augmentations,
lies in the interval $[\voct, 10\vtet)$,
where $\vtet$ is the volume of the regular hyperbolic ideal tetrahedron,
and \cite[Prop 3.7]{kwon2020} shows that these bounds are sharp.
We show (\thmref{t:main}) that the set of FAL volume densities is
dense in $[2\voct, 10\vtet)$, but discrete in $[\voct,2\voct)$,
a stark departure from Burton's result.





We first recall some background material related to
fully augmented links (FALs) in the 3-sphere $S³$.
We refer the reader to \cite{purcell-intro}, \cite{kwon2020}
for more details.
Since we are only interested in the volumes of FALs,
and the volume of a FAL is the same with or without half-twists
(see \cite{adams}), we will only consider FALs with no half-twists.


Let $K$ be a link in $S³$ with prime link diagram $D_K$ in $S²$,
and let $L$ be a FAL obtained by fully augmenting $D_K$;
it is known (e.g. \cite{purcell-intro}) that $L$ is hyperbolic.
Upon removing all full twists, $L$ becomes a union of two unlinks
$L'$ and $L''$, where $L'$ consists of the augmentation circles,
and $L''$ consists of circles contained in the projection surface $S² ⊆ S³$.
Let $a(L)$ denote the number of augmentation circles,
i.e.\ $a(L)$ is the number of components of $L'$.
Since $L$ is symmetric under reflection across $S² ⊆ S³$,
it follows that $Σ := S² \backslash L$ is a totally geodesic surface.
Cutting $S³ \backslash L$ along $Σ$,
we obtain two 3-manifolds $N_L,N_L'$ with totally geodesic boundary $Σ$.
We can picture $N_L$ (and $N_L'$) as a ball $B³$
with several circles (corresponding to $L''$) removed from the boundary $∂B³$
and $a(L)$ closed arcs (with endpoints on the boundary) removed from $B³$.


\begin{definition}{\cite[Def 3.1]{kwon2020}}
\label{d:vol-dens}
Let $L$ be a hyperbolic FAL.
Denote by $\vol(L)$ the volume of its complement,
and by $a(L)$ the number of augmentations on $L$.
We define its \emph{FAL volume density}, denoted $\vd(L)$,
as the ratio of volume to number of augmentations,
i.e.\ $\vd(L) := \vol(L) / a(L)$.
We define its \emph{modified FAL volume density}
to be $\vdm(L) := \vol(L) / (a(L) - 1)$.
\end{definition}

\begin{definition}
The \emph{FAL volume density spectrum} is the closure (in $\RR$)
of the set\footnote{
In \cite{kwon2020}, volume density spectrum refers to this set,
and not its closure;
here we follow \cite{ckp-max}, \cite{burton} instead,
where the volume density spectrum refers to the \emph{closure}
of the set of volume densities.
}
of FAL volume densities of FALs in $S³$.
\end{definition}

\begin{proposition}{\cite[Prop 3.6,3.7]{kwon2020}}
The FAL volume density spectrum lies in $[\voct,10\vtet]$,
and these bounds are sharp.
\end{proposition}

The fully augmented figure eight knot, $L_{4₁}$,
achieves the lower bound, i.e.\ $\vd(L_{4₁}) = \voct$,
while $10\vtet$ can be approached by a sequence of FALs
that Følner converges to the infinite fully augmented square weave.
Note that the upper bound is not achievable.




\begin{theorem}
\label{t:main}
The FAL volume density spectrum
is discrete in the range $[\voct, 2\voct)$
and dense in the range $[2\voct,10\vtet)$.
\end{theorem}
\begin{proof}
Let $L$ be a FAL in $S³$.
For the discreteness result,
we use a similar method to the proof of Proposition 2.14 in \cite{kwon2020}.
In \cite{miyamoto}, Miyamoto showed that
if $N$ is a hyperbolic 3-manifold with totally geodesic boundary,
then $\vol(N) ≥ - \voct \chi(N)$,
with equality exactly when $N$ decomposes into regular ideal octahedra.
We apply this result to $N = N_L$ and $N_L'$.
Recall that $N_L$ is obtained from $B³$ by removing $a(L)$ closed arcs
and some circles from the boundary,
so $\chi(N_L) = 1 - a(L)$, and similarly for $N_L'$.
Thus $\vol(S³ \backslash L) = \vol(N_L) + \vol(N_L')
≥ 2 (a(L) - 1) \voct$, and
$\vd(L) ≥ 2\voct ⋅ (a(L) - 1) / a(L)$,
which is an expression that increases in value as $a(L)$ increases.
Since there are only finitely many FALs with at most a 
given number of augmentations, say $a(L) ≤ n$,
it follows that the set of FAL volume densities,
when restricted to $[\voct, 2\voct ⋅ (n - 1) / n)$, is finite,
and discreteness follows immediately.

Now we prove denseness of FAL volume densities
in the range $[2\voct, 10\vtet)$.
Given two FALs $L₁, L₂$,
and an augmentation circle $C₁ ⊆ L₁, C₂ ⊆ L₂$ from each,
we can consider their \emph{belted sum}, $L₁ \falsum_{C₁,C₂} L₂$,
as defined in \cite[Fig 3(b)]{adams-belted},
which is obtained by cutting the link diagrams along $C₁$ and $C₂$,
gluing the resulting $(2,2)$-tangles, closing the tangle,
and finally adding an augmentation.
By \cite[Cor 5.2]{adams-belted}, the volumes add under belted sum,
i.e.\ $\vol(L₁ \falsum L₂) = \vol(L₁) + \vol(L₂)$,
while the number of augmentations almost add,
i.e.\ $a(L₁ \falsum L₂) = a(L₁) + a(L₂) - 1$,
or equivalently, $\wda(L₁ \falsum L₂) = \wda(L₁) + \wda(L₂)$
(we omit the subscript in $\falsum_{C₁,C₂}$ as
the volume does not depend on the choice of $C₁$ nor $C₂$).
Thus
\[
\vdm(L₁ \falsum L₂) = \frac{\wda(L₁) ⋅ \vdm(L₁) + \wda(L₂) ⋅ \vdm(L₂)}{\wda(L₁) + \wda(L₂)}
\;\;;\;\;
\wda(L₁ \falsum L₂) = \wda(L₁) + \wda(L₂)
\]

More generally, we can take the belted sum of any finite number of FALs,
so that the FAL volume density $L$ is a weighted average:
\[
\vdm(L₁ \falsum ⋅⋅⋅ \falsum L_k) =
	\frac{\sum \wda(L_i) ⋅ \vdm(L_i)}{\sum \wda(L_i)}
\;\;;\;\;
\wda(L₁ \falsum ⋅⋅⋅ \falsum L_k) = \sum \wda(L_i)
\]

Setting all $L_i = L$, and writing $L^{(k)} = L \falsum ⋅⋅⋅ \falsum L$,
we have $\vdm(L^{(k)}) = \vdm(L)$, $a(L^{(k)}) = k ⋅ \wda(L) + 1$,
so
\[
\lim_{k \to ∞} \vd(L^{(k)})
= \lim_{k \to ∞} \vdm(L^{(k)}) ⋅ \frac{a(L^{(k)})}{\wda(L^{(k)})}
= \vdm(L) ⋅ \lim_{k \to ∞} \frac{k ⋅ \wda(L) + 1}{k ⋅ \wda(L)}
= \vdm(L) \ .
\]
Hence it suffices to show that the set of \emph{modified} FAL volume densities
is dense in $[2\voct, 10\vtet)$.

Consider $L^{(k,l)} = L₁^{\falsum k} \falsum L₂^{\falsum l}$,
a belted sum of $k$ copies of $L₁$ and $l$ copies of $L₂$.
We have
\[
\vdm(L^{(k,l)}) = \frac{k ⋅ \wda(L₁) ⋅ \vdm(L₁) + l ⋅ \wda(L₂) ⋅ \vdm(L₂)}
{k ⋅ \wda(L₁) + l ⋅ \wda(L₂)}
\]
Then for any $0 < α < 1$, we can approach the value
$α ⋅ \vdm(L₁) + (1 - α) ⋅ \vdm(L₂)$
by choosing a sequence of rational numbers $\{k_i / l_i\}_{i = 1,...}$
converging to $\wda(L₂) ⋅ α / \wda(L₁) ⋅ (1 - α)$,
so that $\lim_{i \to ∞} k_i ⋅ \wda(L₁) / l_i ⋅ \wda(L₂) = α / (1 - α)$
and hence $\lim_{i \to ∞} \vdm(L^{(k_i,l_i)}) =
α ⋅ \vdm(L₁) + (1 - α) ⋅ \vdm(L₂)$.
Thus the set of modified FAL volume densities
is dense in the interval $[\vdm(L₁), \vdm(L₂)]$.

To conclude the proof, we take $L₁ = L_{4₁}$,
the fully augmented figure-eight knot, which has $\vdm(L₁) = 2\voct$,
and we take $L₂$ to be an element in a sequence of FALs
whose modified FAL volume density approaches $10\vtet$;
for such a sequence, we may take $L₃,L₄,...$ whose
(non-modified) FAL volume density approaches $10\vtet$,
which exists by \cite[Prop 3.7]{kwon2020},
since the number of augmentations $a(L₃),a(L₄),...$
must converge to $∞$, so that
$\lim_{i \to ∞} \vdm(L_i) - \vd(L_i)
= \lim_{i \to ∞} \vd(L_i)/(a(L_i) - 1) = 0$,
\end{proof}

\bibliography{biblio}{}
\bibliographystyle{plain}



\end{document}